\documentclass[12pt,twoside]{amsart}
\usepackage{geometry}
\usepackage{amssymb,amsmath,amsthm, amscd, enumerate, mathrsfs}
\usepackage{graphicx, hhline}
\usepackage[all]{xy}
\usepackage{tikz}

\usepackage{color}

\usepackage[dvipdfmx]{hyperref}
\usepackage{mathtools}
\mathtoolsset{showonlyrefs}

\title{On toric foliated pairs}
\author{Osamu Fujino and Hiroshi Sato} 
\date{2025/3/11, version 0.13}
\address{Department of 
Mathematics, Graduate School of Science, 
Kyoto University, Kyoto 606-8502, Japan}
\email{fujino@math.kyoto-u.ac.jp}
\address{Department of Applied Mathematics, 
Faculty of Sciences, Fukuoka University, 
8-19-1, Nanakuma, Jonan-ku, Fukuoka 814-0180, Japan}
\email{hirosato@fukuoka-u.ac.jp}

\newcommand{\Supp}[0]{{\operatorname{Supp}}}
\newcommand{\rank}[0]{{\operatorname{rank}}}
\newcommand{\mult}[0]{{\operatorname{mult}}}
\newcommand{\codim}[0]{{\operatorname{codim}}}
\newcommand{\Exc}[0]{{\operatorname{Exc}}}
\newcommand{\NE}[0]{{\operatorname{NE}}}
\newtheorem{thm}{Theorem}[section]
\newtheorem*{claim}{Claim}
\newtheorem{cor}[thm]{Corollary}
\newtheorem{lem}[thm]{Lemma}

\theoremstyle{definition}
\newtheorem{defn}[thm]{Definition}
\newtheorem*{ack}{Acknowledgments} 
\newtheorem{rem}[thm]{Remark}
\newtheorem{step}{Step}
\newtheorem{ex}[thm]{Example}

\makeatletter

\@addtoreset{equation}{section}
\makeatother
\begin{document}

\maketitle 

\begin{abstract}
We discuss lengths of extremal rational 
curves, Fujita's freeness, 
and the Kodaira vanishing theorem for log canonical 
toric foliated pairs. 
\end{abstract}

\tableofcontents 

\section{Introduction}\label{a-sec1} 
The basics of toric foliations and toric foliated minimal model program 
were already studied in \cite{pang}, 
\cite[Section 10]{spicer}, \cite{wang}, \cite{chang-chen}, and so on. 
In \cite{fujino-sato2}, we discussed lengths of extremal rational curves 
for toric foliations on projective $\mathbb Q$-factorial toric 
varieties. 
This paper is a continuation of \cite{fujino-sato2} and 
is obviously a generalization of \cite{fujino}. 
Throughout this paper, 
we will work over $\mathbb C$, the field of complex 
numbers. 
The following theorem is a log canonical generalization of 
\cite[Theorem 1.3]{fujino-sato2} or is 
a generalization of \cite[Theorem 0.1]{fujino} for 
toric foliated pairs. 
Note that our approach in this 
paper is based on the toric Mori 
theory (see \cite{reid}, \cite[Chapter 14]{matsuki}, 
\cite{fujino}, and \cite{fujino-sato}). 

\begin{thm}[Lengths of extremal rational curves 
for toric foliated pairs]\label{a-thm1.1}
Let $X$ be a projective 
{\em{(}}not necessarily $\mathbb Q$-factorial{\em{)}} 
toric variety and let $(\mathscr F, \Delta)$ be a log canonical 
toric 
foliated pair on $X$ 
with $\rank \mathscr F=r$.  
Then 
\begin{equation}
l_{(\mathscr F, \Delta)}(R):
=\min _{[C]\in R}\{-(K_{\mathscr F}+\Delta)\cdot C\}\leq r+1
\end{equation}
holds for every extremal ray $R$ of 
the Kleiman--Mori cone 
$\overline{\NE}(X)=
\NE(X)$. Moreover, 
if $l_{(\mathscr F, \Delta)}(R)>r$ 
holds for some extremal ray $R$ of 
$\NE(X)$, then the contraction morphism 
$\varphi_R\colon X\to Y$ associated to 
$R$ is a $\mathbb P^r$-bundle over $Y$. 
In this case, $\mathscr F=\mathscr T_{X/Y}$ holds, where 
$\mathscr T_{X/Y}$ is the relative 
tangent sheaf of $\varphi_R\colon X\to Y$, and the sum of 
the coefficients of $\Delta$ is less than one. 
In particular, the foliation $\mathscr F$ is locally free. 
\end{thm}

We note that we have already treated Theorem \ref{a-thm1.1} 
under the extra assumption that $X$ is 
$\mathbb Q$-factorial and $\Delta=0$ in \cite[Theorem 1.3]{fujino-sato2}. 
By Theorem \ref{a-thm1.1}, 
we have the cone theorem for log canonical 
toric foliated pairs. 

\begin{cor}[Cone theorem for toric foliated 
pairs]\label{a-cor1.2} 
Let $(\mathscr F, \Delta)$ be a log canonical toric foliated 
pair on a projective toric variety $X$ with 
$\rank \mathscr F=r$. 
Then we have 
\begin{equation}
\overline{\NE}(X)=\NE(X)=
\sum_i \mathbb R_{\geq 0}
[C_i]
\end{equation}
where $C_i$ is a torus invariant curve with $-(K_{\mathscr F}+
\Delta)\cdot C_i\leq r+1$ for every $i$. 
Let $R$ be an extremal ray of $\NE(X)$. 
Then we can choose $C_i$ such that 
$-(K_{\mathscr F}+\Delta)\cdot C_i\leq r$ with 
$R=\mathbb R_{\geq 0}[C_i]$ unless 
the associated contraction $\varphi_R\colon X\to Y$ 
is a $\mathbb P^r$-bundle 
over $Y$, 
$\mathscr F=\mathscr T_{X/Y}$, and the sum of 
the coefficients of $\Delta$ is less than one. 
\end{cor}

Corollary \ref{a-cor1.2} is almost obvious by Theorem \ref{a-thm1.1}. 
It is a generalization of the cone theorem for toric varieties 
established in \cite[Theorem 0.1]{fujino}. 
More precisely, if $\rank \mathscr F=\dim X$, 
then Theorem \ref{a-thm1.1} and Corollary \ref{a-cor1.2} 
recovers \cite[Theorem 0.1]{fujino}. 
Fujita's freeness for 
log canonical toric foliated pairs is an easy consequence 
of the cone theorem:~Corollary \ref{a-cor1.2}. 

\begin{thm}[Fujita's freeness for 
toric foliated pairs]\label{a-thm1.3}
Let $(\mathscr F, \Delta)$ be a log canonical toric foliated 
pair on a projective 
toric variety $X$. 
Let $r$ denote the rank of $\mathscr F$. 
Let $H$ be a Cartier divisor on $X$ such that 
$\left( H-(K_{\mathscr F}+\Delta)\right)\cdot C\geq r+1$ holds 
for every torus invariant curve $C$ on $X$. 
Then the complete linear system 
$|H|$ is basepoint-free. 
\end{thm}

For toric foliations on smooth 
projective toric varieties, we have the following statement 
on Fujita's freeness. 

\begin{thm}[Fujita's freeness for toric foliations]\label{a-thm1.4}
Let $\mathscr F$ be a toric foliation with $\rank \mathscr F=r$ 
on a smooth projective toric variety $X$. 
Let $A$ be an ample Cartier divisor on $X$. 
Then $|K_{\mathscr F}+(r+1)A|$ is basepoint-free. 
Moreover, $|K_{\mathscr F}+rA|$ is basepoint-free unless 
$X$ has a $\mathbb P^r$-bundle structure 
$\varphi\colon X\to Y$, $\mathscr F=\mathscr T_{X/Y}$, 
and $A\cdot \ell=1$ for a line $\ell$ in a fiber of $\varphi\colon X\to Y$. 
\end{thm}

If $\rank \mathscr F=\dim X$ in Theorem \ref{a-thm1.4}, 
then it is nothing but the original version of 
Fujita's freeness for smooth projective toric varieties. 
Since any ample Cartier divisor on a smooth 
projective toric variety is very ample, 
we have the following statement on Fujita's very ampleness. 

\begin{thm}[Fujita's very ampleness for toric foliations]\label{a-thm1.5}
Let $\mathscr F$ be a toric foliation with $\rank \mathscr F=r$ 
on a smooth projective toric variety $X$. 
Let $A$ be an ample Cartier divisor on $X$. 
Then $|K_{\mathscr F}+(r+2)A|$ is very ample. 
Moreover, $|K_{\mathscr F}+(r+1)A|$ is very ample unless 
$X$ has a $\mathbb P^r$-bundle structure 
$\varphi\colon X\to Y$, $\mathscr F=\mathscr T_{X/Y}$, 
and $A\cdot \ell=1$ for a line $\ell$ in a fiber of $\varphi\colon X\to Y$. 
\end{thm}

Finally, although we do not treat any applications in this paper, we 
show that the Kodaira vanishing theorem holds 
for log canonical toric foliated pairs. 

\begin{thm}[Kodaira's vanishing theorem for 
toric foliated pairs]\label{a-thm1.6}
Let $(\mathscr F, \Delta)$ be a log 
canonical toric foliated pair on a projective toric variety $X$. 
Let $L$ be a $\mathbb Q$-Cartier Weil divisor on $X$ such that 
$L-(K_{\mathscr F}+\Delta)$ is ample. 
Then $H^i(X, \mathcal O_X(L))=0$ holds for every positive integer $i$. 
\end{thm}

It is a special case of the vanishing theorems 
established in \cite{fujino3}. 

\begin{ack}\label{a-ack}
The first author was partially 
supported by JSPS KAKENHI Grant Numbers 
JP20H00111, JP21H04994, JP23K20787. 
The second author was partially
supported by JSPS KAKENHI Grant Number 
JP24K06679. 
The authors would like to thank Fanjun Meng for pointing 
out that a conjecture in the original version is obviously wrong. 
They also would like to thank the referee very much for useful comments. 
\end{ack}

In this paper, we will use the same notation 
as in \cite{fujino-sato2}. We will freely use 
the basic definitions and results in \cite{fujino-sato2}. 
For the details of 
toric varieties, see \cite{oda1}, \cite{oda2}, 
\cite{fulton}, and \cite{cls}. 
For basic definitions and 
results of the theory of minimal models, 
see \cite{fujino-fundamental} and \cite{fujino-foundations}. 

\section{Preliminaries}\label{a-sec2}

In this section, we collect some definitions and results 
for the reader's convenience. 
Let us start with the definition of {\em{foliations}} on 
normal algebraic varieties. 

\begin{defn}[Foliations and toric foliations]\label{a-def2.1}
A {\em{foliation}} on a normal algebraic variety 
$X$ is a saturated subsheaf 
$\mathscr F\subset \mathscr T_X$ that is closed under the Lie bracket, 
where $\mathscr T_X$ is  the tangent 
sheaf of $X$. 
We note that the {\em{rank}} of the foliation $\mathscr F$ 
means the rank of the coherent sheaf $\mathscr F$. 

We further assume that $X$ is toric. 
Then a foliation $\mathscr F$ on $X$ is called {\em{toric}} if 
the sheaf $\mathscr F$ is torus equivariant. 
\end{defn}

The following result on toric foliations is a starting point of 
\cite{fujino-sato2} and this paper. 

\begin{thm}[{see \cite{pang}}]\label{a-thm2.2}
Let $X=X(\Sigma)$ be a $\mathbb Q$-factorial toric variety 
with its fan $\Sigma$ in the lattice $N\simeq \mathbb Z^n$. 
Then there exists a one-to-one correspondence 
between the set of toric foliations on $X$ and 
the set of complex vector subspaces $V\subset N_{\mathbb C}:=N\otimes 
_{\mathbb Z} \mathbb C\simeq \mathbb C^n$. 

Let $\mathscr F_V$ be the toric foliation associated to 
a complex vector subspace $V\subset N_{\mathbb C}$. 
Then 
\begin{equation}
K_{\mathscr F_V}:=-c_1(\mathscr F_V)=-\sum _{\rho\subset V}D_\rho 
\end{equation} holds, that is, the first Chern class of 
$\mathscr F_V$ is $\sum _{\rho\subset V}D_\rho$, 
where $D_\rho$ is the torus invariant prime divisor corresponding 
to the one-dimensional cone $\rho$ in $\Sigma$. 
In particular, we 
have 
\begin{equation}
K_{\mathscr F_V}=K_X+\sum _{\rho\not\subset V}D_\rho. 
\end{equation}
We note that $\rank \mathscr F=\dim _{\mathbb C} V$. 
\end{thm}

In this paper, we are mainly interested in the case 
where $X$ is not $\mathbb Q$-factorial. 

\begin{rem}\label{a-rem2.3} 
Let $\mathscr F$ be a toric 
foliation on a (not necessarily $\mathbb Q$-factorial) toric 
variety $X$. 
Then we can define $K_{\mathscr F}$ as follows. 
We consider the smooth locus $X_{\mathrm{sm}}$ of $X$. 
Since $X_{\mathrm{sm}}$ is a smooth toric variety, 
we can define $K_{\mathscr F}$ on $X_{\mathrm{sm}}$ as 
in Theorem \ref{a-thm2.2}. 
Since $\codim _X(X\setminus X_{\mathrm{sm}})\geq 2$, we can extend 
it to the whole $X$ and obtain a well-defined torus 
invariant Weil divisor $K_{\mathscr F}$ on 
$X$. 
\end{rem}

For the details of toric 
foliations, see \cite{pang}, \cite{wang} and \cite{chang-chen}. 
We make a remark for the reader's convenience. 

\begin{rem}[{see \cite[Corollary 3.3]{chang-chen}}]\label{a-rem2.4}
Let $\mathscr F_V$ be the toric foliation associated to a complex 
vector subspace 
$V\subset N_{\mathbb C}$. 
Then $D_\rho$ is $\mathscr F_V$-invariant 
if and only if $\rho \not \subset V$. 
\end{rem}

Let us introduce the notion of {\em{torically 
log canonical toric foliated pairs}}. 
We note that if $\rank \mathscr F=\dim X$ in Definition \ref{a-def2.5} then 
it is nothing but the usual definition of toric log canonical pairs. 

\begin{defn}[{Torically log canonical toric foliated pairs}]\label{a-def2.5} 
A {\em{toric foliated pair}} $(\mathscr F, \Delta)$ 
on a toric variety $X$ consists of 
a toric foliation $\mathscr F$ and 
an effective torus invariant $\mathbb R$-divisor 
$\Delta$ on $X$ such that 
$K_{\mathscr F}+\Delta$ is $\mathbb R$-Cartier. 
Let $\pi\colon \widetilde X\to X$ be a proper 
birational toric morphism 
of toric varieties. 
Then we can write 
\begin{equation}
K_{\widetilde {\mathscr F}} +\pi^{-1}_* \Delta
=\pi^*(K_{\mathscr F}+\Delta)+\sum _E 
a(E, \mathscr F, \Delta)E
\end{equation} 
where $\widetilde {\mathscr F}$ is the 
induced foliation on $\widetilde X$ and 
the sum is over all $\pi$-exceptional divisors $E$. 
We call $a(E, \mathscr F, \Delta)$ the {\em{discrepancy}} 
of $E$ with respect to $(\mathscr F, \Delta)$. 
We put $\iota(E)=0$ if $E$ is 
$\widetilde {\mathscr F}$-invariant 
and $\iota (E)=1$ otherwise. 
We say that 
the pair $(\mathscr F, \Delta)$ is {\em{torically log canonical}} 
if $a(E, \mathscr F, \Delta)\geq -\iota(E)$ for 
any proper birational {\em{toric}} morphism $\pi\colon 
\widetilde X\to X$ and 
for any $\pi$-exceptional prime divisor 
$E$ on $\widetilde X$. 
\end{defn}

Although the following lemma is easy to prove, 
it is very important. 

\begin{lem}\label{a-lem2.6}
A toric foliated pair 
$(\mathscr F, \Delta)$ is {\em{torically log canonical}} 
if and only if $\Supp \Delta\subset \Supp K_{\mathscr F}$ 
and the coefficients of $\Delta$ are in $[0, 1]$. 
\end{lem}

We prove Lemma \ref{a-lem2.6} for the sake of completeness. 

\begin{proof}
We assume that 
$\mathscr F$ is the toric foliation associated to 
a complex vector subspace $V\subset N_{\mathbb C}$. 
Then we have 
\begin{equation}
K_{\mathscr F}=K_X+\sum _{\rho\not\subset V}D_\rho 
\end{equation} 
by Theorem \ref{a-thm2.2} and Remark \ref{a-rem2.3}. 
We put 
\begin{equation}
\Delta=\sum _\rho b_\rho D_\rho 
\end{equation} 
with $b_\rho \geq 0$. 
Hence we have 
\begin{equation}\label{a-eq2.1}
K_{\mathscr F}+\Delta=K_X+\sum _{\rho\not\subset V} D_\rho +\sum _\rho 
b_\rho D_{\rho}. 
\end{equation} 
By definition, we can easily see 
that $(\mathscr F, \Delta)$ is torically log canonical 
if and only if 
\begin{equation}
\left(X, \sum _{\rho\not\subset V} D_\rho +\sum _\rho 
b_\rho D_{\rho}\right)
\end{equation} 
is log canonical in the usual sense. 
Thus the pair $(\mathscr F, \Delta)$ is 
torically log canonical if and only if 
$\Supp \Delta\subset \Supp K_{\mathscr F}$ and 
the coefficients of $\Delta$ are in $[0, 1]$. 
\end{proof}

Although it is nontrivial, we see 
that our definition of torically log canonical 
toric foliated pairs coincides with the log 
canonicity of toric foliated pairs in the usual sense. 
In this paper, Definition \ref{a-def2.5} and Lemma \ref{a-lem2.6} 
are sufficient for our purposes.
 
\begin{thm}[{\cite[Proposition 4.31]{chang-chen}}]\label{a-thm2.7} 
Let $(\mathscr F, \Delta)$ be a toric foliated pair. 
Then $(\mathscr F, \Delta)$ is torically log canonical if and only if 
$(\mathscr F, \Delta)$ is log canonical in the usual sense for foliated 
pairs. 
\end{thm}
\begin{proof} 
This follows from Lemma \ref{a-lem2.6} and 
\cite[Proposition 4.31]{chang-chen}. 
\end{proof}

For the precise definition of 
{\em{log canonical foliated pairs}}, 
see, for example, \cite[Definition 4.1]{chang-chen}. 
By Theorem \ref{a-thm2.7}, we can simply say that a 
toric foliated pair $(\mathscr F, \Delta)$ is 
log canonical when it is torically log canonical. 
We close this section with a remark on the minimal model 
program. 

\begin{rem}\label{a-rem2.8}
Let $(\mathscr F, \Delta)$ be a log canonical 
toric foliated pair on a projective $\mathbb Q$-factorial 
toric variety $X$. 
Then we can run the minimal model program with respect 
to $K_{\mathscr F}+\Delta$ 
(see, for example, \cite{reid}, 
\cite{matsuki}, \cite{fujino}, and \cite{fujino-sato}). 
By \eqref{a-eq2.1} in the proof of Lemma \ref{a-lem2.6}, 
we see that the log canonicity of $(\mathscr F, \Delta)$ is preserved 
by the above minimal model program. 
\end{rem}

\section{Lemmas on projective bundles}\label{a-sec3}

In this section, we prepare some lemmas on projective 
bundles over curves for the proof of Theorem \ref{a-thm1.1}. 
Let us start with an easy lemma on projective bundles over 
a smooth rational curve. 

\begin{lem}\label{a-lem3.1}
Let $\pi\colon X\to Y$ be a $\mathbb P^r$-bundle over $\mathbb P^1$. 
We write 
\begin{equation}
\pi\colon X=\mathbb P_{\mathbb P^1}(\mathcal O_{\mathbb P^1}
\oplus \mathcal O_{\mathbb P^1}(c_1)\oplus \cdots 
\oplus \mathcal O_{\mathbb P^1}(c_r))\to \mathbb P^1
\end{equation} 
with $c_1\leq \cdots \leq c_r$. 
Note that $\pi\colon X\to Y$ is toric. 
If there exists an extremal ray $R$ of $\NE(X)$ such that 
$K_{X/Y}\cdot R=0$, 
then $c_1=\cdots =c_r=0$, 
that is, $X=\mathbb P^r\times \mathbb P^1$ and $\pi$ is the 
second projection. 
\end{lem}
\begin{proof}
Since $\pi\colon X\to Y$ is a $\mathbb P^r$-bundle, we have 
\begin{equation}
\mathcal O_X(K_{X/Y})=\pi^*\mathcal O_{\mathbb P^1}\left(
\sum _{i=1}^r c_i\right) \otimes \mathcal O_X(-(r+1)). 
\end{equation}
Note that $\NE(X)$ is spanned by 
two extremal rays. 
One extremal ray corresponds to the projection $\pi\colon X\to Y$. 
Therefore, $K_{X/Y}$ is negative on it. 
By assumption, $K_{X/Y}\cdot C\leq 0$ holds for every horizontal 
torus invariant curve $C$ on $X$. 
This implies $c_1=\cdots =c_r=0$. 
Thus $X=\mathbb P^r\times \mathbb P^1$ and 
$\pi\colon X\to Y$ is the second projection. 
We finish the proof. 
\end{proof}

Lemma \ref{a-lem3.2} is a slight generalization of 
Lemma \ref{a-lem3.1}. 

\begin{lem}\label{a-lem3.2}
Let $\pi\colon X\to Y$ be a $\mathbb P^r$-bundle 
over $\mathbb P^1$ and let $\Delta$ be a torus invariant 
horizontal effective $\mathbb R$-divisor on $X$ 
such that every coefficient of $\Delta$ is less than one. 
If there exists an extremal ray $R$ of $\NE(X)$ such that 
$(K_{X/Y}+\Delta)\cdot R=0$, 
then $X=\mathbb P^r\times \mathbb P^1$ and $\pi$ is the 
second projection. 
\end{lem}
\begin{proof}
As in the proof of Lemma \ref{a-lem3.1}, 
we write 
\begin{equation}
\pi\colon X=\mathbb P_{\mathbb P^1}(\mathcal O_{\mathbb P^1}(c_0)
\oplus \mathcal O_{\mathbb P^1}(c_1)\oplus \cdots 
\oplus \mathcal O_{\mathbb P^1}(c_r))\to \mathbb P^1
\end{equation} 
with $0=c_0\leq c_1\leq \cdots \leq c_r$.
By assumption, we can write 
\begin{equation}
\Delta=\sum _{i=0}^r b_i H_i
\end{equation} 
with $b_i\in [0, 1)$ such that 
\begin{equation}
\mathcal O_X(H_i)\simeq \mathcal O_X(1)\otimes \pi^*\mathcal O_{\mathbb P^1}
(-c_i)
\end{equation} 
for every $i$. 
Let $P$ be a point of $Y=\mathbb P^1$. 
Then $K_{X/Y}+\Delta$ is $\mathbb R$-linearly equivalent to 
\begin{equation}
\pi^*{\left(\sum_{i=0}^r (1-b_i)c_i P\right)}+
\left(-(r+1)+\sum _{i=0}^rb_i\right)H_0. 
\end{equation} 
As in the proof of Lemma \ref{a-lem3.1}, 
$(K_{X/Y}+\Delta)\cdot C\leq 0$ holds for every horizontal 
torus invariant curve $C$ on $X$. 
This implies that 
\begin{equation}
\sum _{i=0}^r(1-b_i)c_i\leq 0 
\end{equation} holds. 
Hence we obtain $c_0=c_1=\cdots =c_r=0$. 
This is what we wanted. 
\end{proof}

The final lemma in this section is similar to 
Lemma \ref{a-lem3.2} above. 
However, we note that $\pi\colon X\to Y$ is not toric 
when $Y\ne \mathbb P^1$. 

\begin{lem}\label{a-lem3.3}
Let $Y$ be a smooth projective curve and let $\mathcal L_i$ 
be a line bundle on $Y$ for every $i$. We consider a 
$\mathbb P^r$-bundle $\pi\colon 
X:=\mathbb P_Y(\mathcal L_0\oplus \cdots \oplus 
\mathcal L_r)\to Y$ over $Y$. 
Let $H_i$ be the horizontal divisor on $X$ corresponding 
to 
\begin{equation}
\bigoplus _{j=0}^r \mathcal L_j\to \bigoplus _{j\ne i} \mathcal L_j
\end{equation} 
for every $i$. 
We put $\Delta=\sum _{i=0}^r b_i H_i$ such that 
$b_i\in[0, 1)$ for every $i$. 
Assume that there exists an extremal ray $R$ of 
$\overline{\NE}(X)$ such that 
$(K_{X/Y}+\Delta)\cdot R=0$. 
Then $\deg \mathcal L_i=\deg \mathcal L_0$ holds 
for every $i$. 
In particular, if $\deg \mathcal L_0=0$, 
then $\deg \mathcal L_i=0$ holds for every $i$. 
Let $C_i$ be the section of $\pi\colon X\to Y$ corresponding to 
\begin{equation}
\bigoplus _{j=0}^r\mathcal L_j\to \mathcal L_i
\end{equation} 
for every $i$. 
Then the numerical equivalence class of $C_i$ is in $R$ for every $i$. 
\end{lem}

\begin{proof}
Note that 
\begin{equation}
\mathcal O_X(H_i)\simeq \mathcal O_X(1)\otimes \pi^*\mathcal L^{\otimes -1}_i
\end{equation} 
holds for every $i$ and 
that 
\begin{equation}
\mathcal O_X(K_{X/Y})\simeq \pi^*\left(\bigotimes _{i=0}^r \mathcal L_i\right)
\otimes 
\mathcal O_X(-(r+1)). 
\end{equation}
Hence we can easily check this lemma by modifying the 
proof of Lemma \ref{a-lem3.2} suitably. 
\end{proof}

\section{Proof of Theorem \ref{a-thm1.1}}\label{a-sec4}

This section is the main part of this paper. 
Here we give a proof of Theorem \ref{a-thm1.1}. 

\begin{proof}[Proof of Theorem \ref{a-thm1.1}]
In Step \ref{a-step1}, we will prove Theorem \ref{a-thm1.1} under 
the extra assumption that $X$ is $\mathbb Q$-factorial. 
Step \ref{a-step1} is essentially the same as the proof of 
\cite[Theorem 1.3]{fujino-sato2}. 
Hence we will only explain how to modify it. 
Then, in Step \ref{a-step2}, we will treat the case where 
$X$ is not $\mathbb Q$-factorial. 
Step \ref{a-step2} is completely new. 
In our proof in Step \ref{a-step2}, 
we have to treat non-toric varieties. 

\begin{step}\label{a-step1}
We assume that 
$\mathscr F$ is the toric foliation associated to 
a complex vector subspace $V\subset N_{\mathbb C}$. 
Then we can write 
\begin{equation}
\Delta=\sum _{\rho \subset V} b_\rho D_\rho 
\end{equation} 
with $b_\rho \in [0, 1]$ 
and 
\begin{equation}
K_{\mathscr F}+\Delta=K_X+\sum _{\rho\not\subset V} D_\rho +\sum _{\rho 
\subset V}b_\rho D_{\rho}
\end{equation} 
since $(\mathscr F, \Delta)$ is log canonical 
(see Lemma \ref{a-lem2.6} and Theorem \ref{a-thm2.7}). 
We assume that $l_{(\mathscr F, \Delta)}(R)>r$ holds. 
From now, we will only explain how to modify the 
proof of 
\cite[Theorem 1.3]{fujino-sato2}. 
Hence we will freely use the same notation as in the proof of 
\cite[Theorem 1.3]{fujino-sato2}. 
We put $b_\rho =1$ for $\rho \not\subset V$ and 
$b_i:=b_{\rho_i}$ with $\rho_i:=\mathbb R_{\geq 0}v_i$ for every $i$. 
By changing the order, we may assume that 
\begin{equation}
(1-b_1)a_1\leq \cdots \leq (1-b_n)a_n\leq (1-b_{n+1})a_{n+1}. 
\end{equation} 
Then we have 
\begin{equation}
-(K_{\mathscr F}+\Delta)\cdot C=\sum_{v_i\in V}
(1-b_i)V(\langle v_i\rangle)\cdot C>r. 
\end{equation}
By the same argument as in the proof of 
\cite[Theorem 1.3]{fujino-sato2}, 
we obtain 
\begin{equation}
a_{n-r+1}v_{n-r+1}+\cdots +a_{n+1}v_{n+1}=0. 
\end{equation}
Then we see that 
\begin{equation}
\begin{split}
r<-(K_{\mathscr F}+\Delta)\cdot V(\mu_{k, n+1})&\leq 
\frac{1}{a_{n+1}}\left(\sum _{i=n-r+1}^{n+1} 
(1-b_i)a_i\right) \frac{\mult (\mu_{k, n+1})}
{\mult (\sigma_k)}\\&\leq (r+1)\frac{\mult (\mu_{k, n+1})}
{\mult (\sigma_k)} 
\end{split}
\end{equation}
holds for every $n-r+1\leq k\leq n$. Then the 
argument in the proof of \cite[Theorem 1.3]{fujino-sato2} works 
without any changes. 
Thus we obtain that 
$\varphi_R\colon X\to Y$ is a $\mathbb P^r$-bundle and 
$\mathscr F=\mathscr T_{X/Y}$. 
In this case, we can easily check that 
the sum of the coefficients of $\Delta$ is less than one 
by $l_{(\mathscr F, \Delta)}(R)>r$. 
\end{step}

\medskip 

From now, we may assume that $X$ is not $\mathbb Q$-factorial. 
It is sufficient to prove that $\varphi_R\colon X\to Y$ is a $\mathbb P^r$-bundle 
with $\mathscr F=\mathscr T_{X/Y}$ 
under the assumption that $l_{(\mathscr F, \Delta)}(R)>r$. 

\begin{step}\label{a-step2}
We take a small projective $\mathbb Q$-factorialization 
$\psi\colon X'\to X$ (see \cite[Corollary 5.9]{fujino}). 
Let $\Delta'$ be the strict transform of $\Delta$ and let $\mathscr F'$ be the 
induced foliation on $X'$. 
By construction, we have 
$K_{\mathscr F'}+\Delta'=\psi^*(K_{\mathscr F}+\Delta)$. 
Let $\varphi_R\colon X\to Y$ be the contraction morphism 
associated to $R$ (see \cite[Theorem 4.5]{fujino-sato}). 
By considering $\varphi_R\circ \psi\colon X'\to Y$, 
we can find an extremal ray $R'$ of $\NE(X')$ such that 
$\psi_*R'=R$ and $l_{(\mathscr F', \Delta')}(R')>r$. 
Since $X'$ is $\mathbb Q$-factorial, 
the associated contraction $\varphi_{R'}\colon X'\to Y'$ is a $\mathbb P^r$-bundle, 
$\mathscr F'$ is the relative tangent sheaf $\mathscr T_{X'/Y'}$, and the sum of 
the coefficients of $\Delta'$ is less than one 
by Step \ref{a-step1}. 
Note that we can write 
\begin{equation}
\varphi_{R'}\colon X'=\mathbb P_{Y'}(\mathcal L_0\oplus 
\mathcal L_1\oplus \cdots \oplus \mathcal L_r)\to Y'
\end{equation} 
with $\mathcal L_0=\mathcal O_{Y'}$ since $\varphi_{R'}\colon 
X'\to Y'$ is toric (see Lemma \ref{a-lem6.1} below).  
We put $E:=\Exc(\psi)$, that is, the exceptional 
locus of $\psi$. 
Then $E$ is a torus invariant closed subset 
of $X'$ with $\codim_{X'}E\geq 2$. 
\begin{claim} 
$E=\varphi^{-1}_{R'}(\varphi_{R'}(E))$. 
\end{claim}
\begin{proof}[Proof of Claim]
Let $Z'$ be the section of $\varphi_{R'}\colon X'\to Y'$ 
corresponding to 
\begin{equation}
\bigoplus _{i=0}^r\mathcal L_i\to \mathcal L_0. 
\end{equation}
We consider $\psi_{Z'}:=\psi|_{Z'}\colon Z'\to Z:=\psi(Z')$. 
Then any positive-dimensional fiber of $\psi_{Z'}$ 
is rationally chain connected since $\psi_{Z'}\colon Z'\to Z$ is toric. 
Let $C$ be a rational curve in a fiber of $\psi_{Z'}$. 
Let $C'$ be the normalization of $\varphi_{R'}(C)$. 
We consider the base change of 
$\varphi_{R'}\colon X'\to Y'$ by $\mathbb P^1\simeq C'\to Y'$. 
Then, by Lemma \ref{a-lem3.2}, we have the following 
commutative diagram: 
\begin{equation}
\xymatrix{
\mathbb P^r &\ar[l]_-{p_1}\ar[d]^-{p_2}\ar[r] 
\mathbb P^r\times C' \ar@/^{18pt}/[rr]^-\alpha& 
\varphi^{-1}_{R'}(\varphi_{R'}(C))
\ar@{^{(}->}[r]\ar[d]& X'\ar[d]^-{\varphi_{R'}}\\ 
&  C'\ar[r] & \varphi_{R'} (C)\ar@{^{(}->}[r] & Y', 
}
\end{equation}
where $p_1$ and $p_2$ are projections. 
Thus, the numerical equivalence class of $\alpha(p^{-1}_1(P))$ 
is independent of $P\in \mathbb P^r$. 
This implies that $\alpha (p^{-1}_1(P))$ is numerically 
equivalent to $C$ for every $P\in \mathbb P^r$. 
Hence $\psi\colon X'\to X$ contracts $\alpha (p^{-1}_1(P))$ to a point for every 
$P\in \mathbb P^r$. 
Thus, we obtain $\varphi^{-1}_{R'}(\varphi_{R'}(C))\subset E$. 
Therefore, we obtain 
\begin{equation}\label{a-eq4.1}
\varphi^{-1}_{R'}(\varphi_{R'}(\Exc(\psi_{Z'})))\subset E. 
\end{equation} 
If $\psi_Z\colon Z'\to Z$ is not birational, 
then $\Exc(\psi_{Z'})=Z'$. 
This implies $X\subset E$ by \eqref{a-eq4.1}. 
This is obviously a contradiction. Hence 
we obtain that $\psi_{Z'}\colon Z'\to Z$ is birational. 
If $\varphi^{-1}_{R'}(\varphi_{R'}(\Exc(\psi_{Z'})))\subsetneq E$, 
then we can take a curve $C$ such that $\psi(C)$ is a point 
with 
\begin{equation}
C\not\subset \varphi^{-1}_{R'}(\varphi_{R'}(\Exc(\psi_{Z'}))). 
\end{equation}
Let $C'$ be the normalization of $\varphi_{R'}(C)$. 
We consider the base change of 
$\varphi_{R'}\colon X'\to Y'$ by $C'\to Y'$. 
Then we have the following commutative diagram: 
\begin{equation}\label{a-eq4.2}
\xymatrix{
\ar[d]\ar[r]_-\beta 
X'_{C'} \ar@/^{18pt}/[rr]^-\alpha& 
\varphi^{-1}_{R'}(\varphi_{R'}(C))
\ar@{^{(}->}[r]\ar[d]& X'\ar[d]^-{\varphi_{R'}}\\ 
  C'\ar[r] & \varphi_{R'} (C)\ar@{^{(}->}[r] & Y', 
}
\end{equation}
where $X'_{C'}\to C'$ is the base change of $X'\to Y'$ by $C'\to Y'$. 
By construction, we can take a curve $C^\flat$ on $X'_{C'}$ such that 
$\beta(C^\flat)=C$. 
Since $X'_{C'}$ is a $\mathbb P^r$-bundle over $C'$, 
the Picard number of $X'_{C'}$ is two. 
Moreover, $X'_{C'}$ has two nontrivial contractions 
$X'_{C'}\to C'$ and $\psi\circ \alpha \colon X'_{C'}\to W$, where 
$W$ is the normalization of $(\psi\circ\alpha)(X'_{C'})$. 
Thus $\overline{\NE}(X'_{C'})$ has an extremal ray $Q$ spanned by $C^\flat$ since 
$C^\flat$ is contracted 
by $\psi\circ\alpha$. 
Note that $C^\flat\cdot (K_{X'_{C'}/C'}+\Delta'_{C'})=0$, 
where $K_{X'_{C'}/C'}+\Delta'_{C'}=\alpha^*(K_{X'/Y'}+\Delta')$. 
By Lemma \ref{a-lem3.3}, we can take a curve $C^\sharp$ on $X'_{C'}$ such that 
the numerical equivalence class of $C^\sharp$ 
is in $Q$ and $\beta(C^\sharp)=\varphi^{-1}_{R'}(\varphi_{R'} (C))\cap Z'$. 
Thus the curve $\varphi^{-1}_{R'}(\varphi_{R'} (C))\cap Z'$ is contracted 
by $\psi$, 
that is, $\varphi^{-1}_{R'}(\varphi_{R'}(C))\cap Z' \subset \Exc(\psi_{Z'})$. 
Then 
\begin{equation}
\varphi^{-1}_{R'}(\varphi_{R'}(C))\cap Z' \subset \Exc(\psi_{Z'})
\subset \varphi^{-1}_{R'}(\varphi_{R'}(\Exc(\psi_{Z'}))).
\end{equation}
Hence we have 
\begin{equation}
C\subset \varphi^{-1}_{R'}(\varphi_{R'}(\Exc(\psi_{Z'}))). 
\end{equation}
This is a contradiction. 
This implies that 
\begin{equation}
E= \varphi^{-1}_{R'}(\varphi_{R'}(\Exc(\psi_{Z'}))). 
\end{equation}
Therefore, we have the desired equality 
$E= \varphi^{-1}_{R'}(\varphi_{R'}(E))$. 
We finish the proof of Claim.  
\end{proof}
We put $\mathcal L_{i, Z'}:=(\varphi^*_{R'}\mathcal L_i)|_{Z'}$ 
for every $i$. 
Let $C$ be any curve on $Z'$ such that 
$\psi_{Z'}(C)$ is a point. 
Let $C'$ be the normalization of $\varphi_{R'}(C)$. 
We consider the commutative diagram \eqref{a-eq4.2} as before. 
Then we can check that $\mathcal L_{i, Z'}\cdot C=0$ for 
every $i$ by applying Lemma \ref{a-lem3.3} to $X'_{C'}\to C'$. 
This implies 
that there exists a line bundle $\mathcal L_{i, Z}$ on $Z$ 
such that $\mathcal L_{i, Z'}=\psi^*_{Z'}\mathcal L_{i, Z}$ 
holds for every $i$ since $\psi_{Z'}\colon Z'\to Z$ is 
a projective birational toric morphism. 
Hence, $\mathcal L_{i, Z'}|_C$ is a trivial line bundle for every $i$. 
Then $X'_{C'}\to C'$, which is 
the base change of $\varphi_{R'}\colon X'\to Y'$ by $C'\to Y'$, is 
the second projection 
\begin{equation}
\mathbb P^r\times C'=\mathbb P_{C'}(\mathcal O_{C'}
\oplus \cdots \oplus \mathcal O_{C'})\to C'.  
\end{equation} 
In particular, we obtain that $\Delta'\cdot C^\dag=0$ holds 
for every curve $C^\dag$ on $X'$ such that 
$\psi(C^\dag)$ is a point. 

We consider $\psi_{Z'}\circ (\varphi_{R'}|_{Z'})^{-1}\colon 
Y'\to Z$. 
By the above observation, 
for any point $x\in X$, we see that 
$(\psi_{Z'}\circ (\varphi_{R'}|_{Z'})^{-1}\circ 
\varphi_{R'}) (\psi^{-1}(x))$ is a point. 
Therefore, there exists a morphism $X\to Z$ and 
we have the following commutative diagram.  
\begin{equation}
\xymatrix{
X'\ar[d]_-{\varphi_{R'}}\ar[rr]^-\psi& & X\ar[d]\ar[d]\\ 
Y'\ar[r]^-{\simeq}& Z' \ar[r]_-{\psi_{Z'}}& Z  
}
\end{equation} 
By this commutative diagram and the observation before, 
we see that every fiber of $X\to Z$ is contracted to a point by $\varphi_R$. 
Thus $\varphi_R\colon X\to Y$ factors through $Z$. 
Since the relative Picard number of $\varphi_R\colon X\to Y$ is one, 
$Z$ is isomorphic to $Y$. 
Hence we have the following commutative diagram 
\begin{equation}
\xymatrix{
X'\ar[d]_-{\varphi_{R'}}\ar[r]^-\psi&  X\ar[d]^-{\varphi_R}\\ 
Y'\ar[r]_-{\psi_{Y'}}& Y 
}
\end{equation} 
and we see that $\mathcal L_i=\psi^*_{Y'}\mathcal M_i$ holds 
for some line bundle 
$\mathcal M_i$ on $Y$ for every $i$. 

We put $X'':=\mathbb P_Y(\mathcal M_0\oplus \cdots \oplus \mathcal M_r)$. 
Then $\varphi_{R'}\colon X'\to Y'$ is the base change of 
$X''\to Y$ by $\psi_{Y'}\colon Y'\to Y$. 
We put $\rho\colon X'\to X''$.  
Then $K_{X'/Y'}+\Delta'=\rho^*(K_{X''/Y}+\Delta'')$ with 
$\Delta'':=\rho_*\Delta'$. By construction,  
$K_{X'/Y'}+\Delta'=\psi^*(K_{X/Y}+\Delta)$. 
We set $B:=(\psi_{Y'}\circ \varphi_{R'})(E)$. 
Then $B$ is a closed subset of $Y$ with $\codim _Y B\geq 2$. 
By Claim and the construction of $\psi_{Y'}$, 
we have $E=(\psi_{Y'}\circ \varphi_{R'})^{-1}(B)$. 
Thus we obtain the following commutative diagram: 
\begin{equation}
\xymatrix{
X'\setminus E\ar[d]_-{\varphi_{R'}}\ar[r]_-\psi^-\sim& 
X\setminus \psi(E)\ar[d]^-{\varphi_R}\\ 
Y'\setminus \varphi_{R'}(E)\ar[r]_-{\psi_{Y'}}^-\sim& Y 
\setminus B. 
}
\end{equation}
Note that $\codim _{X'}E\geq 2$, $\codim _{Y'} \varphi_{R'}(E)\geq 2$, 
and $\codim _X\psi(E)\geq 2$. 
By the above diagram, we can easily check that 
$X''$ and $X$ are isomorphic in codimension one. 
By construction again, $-(K_{X/Y}+\Delta)$ is 
ample over $Y$ and 
$-(K_{X''/Y}+\Delta'')$ is 
also ample over $Y$. 
Hence, $X$ is isomorphic 
to $X''$ over $Y$. 
This is what we wanted. 
\end{step}
We finish the proof of Theorem \ref{a-thm1.1}. 
\end{proof}

\begin{rem}\label{a-rem4.1} 
The authors do not know how to prove 
Theorem \ref{a-thm1.1} 
using only toric geometry when $X$ is not $\mathbb Q$-factorial. 
\end{rem}

We close this section with an example, which 
shows that the estimate in Theorem \ref{a-thm1.1} is sharp. 

\begin{ex}\label{a-ex4.2}
We consider $N=\mathbb Z^2$. 
We put 
\begin{equation}
v_1=\begin{pmatrix}1\\0\end{pmatrix},
v_2=\begin{pmatrix}1\\1\end{pmatrix}, 
v_3=\begin{pmatrix}0\\1\end{pmatrix}, 
\ \text{and} \ \ 
v_4=\begin{pmatrix}-1\\-1\end{pmatrix}. 
\end{equation}
Let us consider the fan $\Sigma$ consisting of 
$\mathbb R_{\geq 0}v_1+\mathbb R_{\geq 0} v_2$, 
$\mathbb R_{\geq 0}v_2+\mathbb R_{\geq 0} v_3$,
$\mathbb R_{\geq 0}v_3+\mathbb R_{\geq 0} v_4$,
$\mathbb R_{\geq 0}v_4+\mathbb R_{\geq 0} v_1$, 
and their faces. 
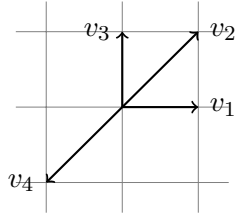
\begin{center}
\begin{tikzpicture}
\draw[step=1.0,very thin, gray] (-1.4,-1.4) grid (1.4,1.4);
\draw[->][thick] (0,0) -- (1, 0) node[right]{\begin{small}$v_1$\end{small}};
\draw[->][thick] (0,0) -- (0, 1) node[left]{\begin{small}$v_3$\end{small}};
\draw[->][thick] (0,0) -- (1, 1) node[right]{\begin{small}$v_2$\end{small}};
\draw[->][thick] (0,0)-- (-1, -1) node[left]{\begin{small}$v_4$\end{small}};
\end{tikzpicture}
\end{center}
Then the toric variety $X:=X(\Sigma)$ is 
a $\mathbb P^1$-bundle 
$\mathbb P_{\mathbb P^1}(\mathcal O_{\mathbb P^1}\oplus \mathcal O_{\mathbb P^1}(1))$ 
over $\mathbb P^1$. 
We put $\rho_i:=\mathbb R_{\geq 0}v_i$ and $D_i:=D_{\rho_i}$ for 
every $i$. 
Then the Kleiman--Mori cone is spanned by $[D_2]$ and $[D_1]=[D_3]$, that is, 
\begin{equation}
\NE(X)=\mathbb R_{\geq 0}[D_2]+\mathbb R_{\geq 0}[D_3]. 
\end{equation}
Let $V$ be the complex vector subspace of $N_{\mathbb C}$ spanned 
by $v_2$. 
Let $\mathscr F_V$ be the associated toric foliation on $X$. 
Then $\rank \mathscr F_V=1$ and $K_{\mathscr F_V}=-D_2-D_4$ 
(see Theorem \ref{a-thm2.2}). 
Similarly, let $W$ be 
the complex vector subspace of $N_{\mathbb C}$ spanned 
by $v_1$ and let 
$\mathscr F_W$ be the associated toric foliation on $X$. 
Then $K_{\mathscr F_W}=-D_1$ and $\rank \mathscr F_W=1$ 
(see Theorem \ref{a-thm2.2}). 
We can directly check that 
\begin{equation}
\begin{cases}
-K_{\mathscr F_V}\cdot D_2=-1\\ 
-K_{\mathscr F_V}\cdot D_3=2
\end{cases}
\end{equation}
and 
\begin{equation}
\begin{cases}
-K_{\mathscr F_W}\cdot D_2=1\\ 
-K_{\mathscr F_W}\cdot D_3=0. 
\end{cases}
\end{equation} 
Note that $\mathbb R_{\geq 0}[D_3]$ corresponds 
to the $\mathbb P^1$-bundle structure of $X$ and 
that $\mathbb R_{\geq 0}[D_2]$ gives a blow-down 
$X=\mathbb P_{\mathbb P^1}(\mathcal O_{\mathbb P^1}\oplus 
\mathcal O_{\mathbb P^1}(1))\to \mathbb P^2$. 
\end{ex}

\section{Proofs of Corollary \ref{a-cor1.2}, 
Theorems \ref{a-thm1.3}, \ref{a-thm1.4}, \ref{a-thm1.5}, and \ref{a-thm1.6}}

In this section, we prove the results in Section \ref{a-sec1}. 

\begin{proof}[Proof of Corollary \ref{a-cor1.2}]
It is well known that $\NE(X)$ is spanned by 
torus invariant curves on $X$ (see, for example, \cite{reid}, 
\cite{matsuki}, \cite{fujino}, and \cite{fujino-sato}). 
In particular, it is a rational polyhedral cone. 
The statement on lengths of extremal rational curves 
follows from Theorem \ref{a-thm1.1}. We finish the proof. 
\end{proof}

Theorems \ref{a-thm1.3} and \ref{a-thm1.4} are 
easy consequences of the cone theorem:~Corollary 
\ref{a-cor1.2}. 

\begin{proof}[Proof of Theorem \ref{a-thm1.3}]
By Corollary \ref{a-cor1.2}, 
$H\cdot R\geq 0$ for every extremal ray $R$ of $\NE(X)$, 
that is, $H$ is a nef Cartier divisor on $X$. 
This implies that $|H|$ is basepoint-free. 
\end{proof}

\begin{proof}[Proof of Theorem \ref{a-thm1.4}]
This follows from Corollary \ref{a-cor1.2}. 
More precisely, we can check the nefness of 
$K_{\mathscr F}+(r+1)A$ and $K_{\mathscr F}+rA$ 
under the given assumptions as in the proof of 
Theorem \ref{a-thm1.3}.  
\end{proof}

We learned the following example from Professor Fanjun Meng. 

\begin{ex}\label{a-ex5.1} 
For any positive integer $n$, we can construct 
a smooth projective surface $Y$ and an ample 
Cartier divisor $H$ on $Y$ such that $|nH|$ is 
not basepoint-free (see \cite[Examples 5.1.18 and 5.2.1]{lazarsfeld}). 
We take an elliptic curve $E$ and a point $P$ of $E$. 
We put $X:=Y\times E$ and $A:=p^*_1H+p^*_2P$, 
where $p_1\colon X\to Y$ and $p_2\colon X\to E$ 
are projections. Then $A$ is an ample Cartier 
divisor on $X$ such that $|nA|$ is not basepoint-free. 
We consider $\pi:=p_1\colon X\to Y$ and 
$\mathscr F:=\mathscr T_{X/Y}$, that is, $\mathscr F$ is the relative 
tangent sheaf of $\pi\colon X\to Y$. Then the 
canonical bundle $\mathcal O_X(K_{\mathscr F})$ 
of $\mathscr F$ is trivial. In this case, 
$|K_{\mathscr F}+nA|$ is not basepoint-free. 
Hence we cannot formulate Fujita's freeness conjecture for foliations naively. 
\end{ex}

Theorem \ref{a-thm1.5} is obvious by Theorem \ref{a-thm1.4}. 

\begin{proof}[Proof of Theorem \ref{a-thm1.5}]
Since $A$ is an ample Cartier divisor on 
a smooth projective toric variety $X$, $A$ is 
very ample (see, for example, \cite[Corollary 2.15]{oda2}). 
Hence we have the desired statement by Theorem \ref{a-thm1.4}. 
\end{proof}

We finally prove the Kodaira vanishing theorem for 
log canonial toric foliated pairs. 

\begin{proof}[Proof of Theorem \ref{a-thm1.6}]
We assume that 
$\mathscr F$ is the toric foliation associated to 
a complex vector subspace $V\subset N_{\mathbb C}$. 
Then we can write 
\begin{equation}
\Delta=\sum _{\rho \subset V} b_\rho D_\rho 
\end{equation} 
with $b_\rho \in [0, 1]$ 
and 
\begin{equation}
K_{\mathscr F}+\Delta=K_X+\sum _{\rho\not\subset V} D_\rho +\sum _{\rho 
\subset V}b_\rho D_{\rho}
\end{equation} 
since $(\mathscr F, \Delta)$ is log canonical 
(see Lemma \ref{a-lem2.6}). 
By assumption, 
\begin{equation}
L-(K_{\mathscr F}+\Delta)=L-\left(K_X+
\sum _{\rho\not\subset V} D_\rho +\sum _{\rho 
\subset V}b_\rho D_{\rho}\right) 
\end{equation} 
is ample. 
By perturbing the coefficients, we can construct an 
effective $\mathbb Q$-divisor $\Delta'$ on $X$ 
such that 
\begin{equation}
\Supp \Delta' =\Supp \left(\sum _{\rho\not\subset V} D_\rho +\sum _{\rho 
\subset V}b_\rho D_{\rho}\right), 
\end{equation} 
every coefficient of $\Delta'$ is 
less than one, and $L-(K_X+\Delta')$ is still ample. 
In this setting, by \cite[Corollary 1.7]{fujino3}, 
we obtain 
\begin{equation}
0=H^i(X, \mathcal O_X(K_X+\lceil L-(K_X+\Delta')\rceil))
=H^i(X, \mathcal O_X(L)) 
\end{equation} 
for every positive integer $i$. 
This is what we wanted. 
\end{proof}

\section{Appendix:~Toric projective bundles}
In this appendix, we give a proof of the following well-known 
result (see \cite[p.41 Remark]{oda1}) 
for the sake of completeness. 
To the best knowledge of the authors, 
we do not find it in the standard literature. 

\begin{lem}[{Toric projective bundles, \cite[p.41 Remark]{oda1}}]\label{a-lem6.1} 
Let $\varphi\colon X\to Y$ be a toric 
morphism of toric varieties such that $\varphi\colon X
\to Y$ 
is a $\mathbb P^r$-bundle. 
Then $X\simeq \mathbb P_Y(\mathcal L_0\oplus\cdots  
\oplus \mathcal L_r)$ for some line bundles $\mathcal L_0, \ldots,  
\mathcal L_r$ on $Y$ and 
$\varphi\colon X\to Y$ 
is isomorphic to the projection $\pi\colon \mathbb P_Y(\mathcal L_0\oplus 
\cdots \oplus \mathcal L_r)
\to Y$. 
\end{lem}
\begin{proof}
Since $X$ is a $\mathbb P^r$-bundle over $Y$, we can write 
$X=\mathbb P_Y(\mathcal E)$ for some vector bundle $\mathcal E$ on $Y$. 
We take a torus invariant Cartier divisor $H$ on $X$ 
such that $\mathcal O_X(H)\simeq \mathcal O_{\mathbb P_Y(\mathcal E)}(1)$. 
Then we have $\varphi_*\mathcal O_X(H)\simeq \mathcal E$. 
Thus, by replacing $\mathcal E$ with $\varphi_*\mathcal O_X(H)$, 
we may assume that $\mathcal E$ is a toric vector bundle on $Y$, 
that is, the torus action on $Y$ lifts to an action on $\mathcal E$ 
and it is linear on the fibers. Let $U$ be any affine 
toric open subset $U$ of $Y$. 
Then it is not difficult to see that $\mathcal E|_U$ is 
isomorphic to $\mathcal O^{\oplus r+1}_U$ as a 
toric vector bundle on $U$ 
(see, for example, \cite[Proposition 2.2]{payne-m}). 
Therefore, the restriction of $\varphi\colon X\to Y$ 
to $U$ is isomorphic to the second projection $\mathbb P^r\times U\to U$. 
Let $h\colon (N, \Sigma)\to (N', \Sigma')$ be a map of 
fans corresponding to $\varphi\colon X\to Y$. 
Let $N''$ be the kernel of $h\colon N\to N'$. 
Without loss of generality, we may assume that 
$N=N''\oplus N'$. 
We fix a $\mathbb Z$-basis $\{n''_1, \ldots, n''_r\}$ of $N''$. 
Since $\varphi\colon X\to Y$ is isomorphic to 
the second projection $\mathbb P^r \times U\to U$ for any 
affine toric open subset $U$ of $Y$, 
we can lift any cone $\sigma'\in \Sigma'$ to a cone 
$\sigma\in \Sigma$. 
Hence we can find $\Sigma'$-linear support functions 
$h_1, \ldots, h_r$ such that 
the map $N'_{\mathbb R}\to N_{\mathbb R}=N''_{\mathbb R}\oplus 
N'_{\mathbb R}$ given by $y\mapsto \left(\sum _{i=1}^rh_i(y)n''_i, y\right)$ 
defines the desired lifts of cones.  
Let $\mathcal L_i$ be the line bundle on $Y$ defined 
by the $\Sigma'$-linear 
support function $h_i$ for every $i$. 
Then, by construction, we can check that 
$X\simeq \mathbb P_Y(\mathcal O_Y\oplus \mathcal L_1\oplus 
\cdots \oplus \mathcal L_r)$ and 
$\varphi\colon X\to Y$ is isomorphic to 
the projection $\pi\colon \mathbb P_Y(\mathcal O_Y\oplus 
\mathcal L_1\oplus \cdots \oplus \mathcal L_r)\to Y$ 
(see Remark \ref{a-rem6.2} below). 
We finish the proof. 
\end{proof}

\begin{rem}[{see \cite[p.124, Remark.(2)]{park}}]\label{a-rem6.2}
The minus sign in \cite[p.59]{oda2} needs to be deleted. 
\end{rem}


\begin{thebibliography}{FjS2}

\bibitem[CC]{chang-chen}
C.-W.~Chang, Y.-A.~Chen, 
On toric foliations, preprint (2023). arXiv:2308.05053 [math.AG]

\bibitem[CLS]{cls}
D.~A.~Cox, J.~B.~Little, H.~K.~Schenck, 
{\em{Toric varieties}}, 
Graduate Studies in Mathematics, 
\textbf{124}. American Mathematical Society, Providence, RI, 2011.

\bibitem[Fj1]{fujino}
O.~Fujino, 
Notes on toric varieties from Mori theoretic viewpoint, 
Tohoku Math. J. (2) \textbf{55} (2003), no. 4, 551--564.

\bibitem[Fj2]{fujino3} 
O.~Fujino, Multiplication maps 
and vanishing theorems for toric varieties, 
Math. Z. \textbf{257} (2007), no. 3, 631--641.

\bibitem[Fj3]{fujino-fundamental}
O.~Fujino, 
Fundamental theorems for the log minimal 
model program, 
Publ. Res. Inst. Math. Sci. \textbf{47} (2011), no. 3, 727--789. 

\bibitem[Fj4]{fujino-foundations}
O.~Fujino, {\em{Foundations of the minimal model program}}, 
MSJ Memoirs, \textbf{35}. Mathematical Society of Japan, Tokyo, 2017. 

\bibitem[FjS1]{fujino-sato} 
O.~Fujino, H.~Sato, 
Introduction to the toric Mori theory, 
Michigan Math. J. \textbf{52} (2004), no. 3, 649--665.

\bibitem[FjS2]{fujino-sato2}
O.~Fujino, H.~Sato, 
A remark on toric foliations, 
Arch. Math. (Basel) \textbf{122} (2024), no. 6, 621--627.

\bibitem[Fl]{fulton}
W.~Fulton, {\em{Introduction to toric varieties}}, 
Annals of Mathematics Studies, \textbf{131}. The William 
H.~Roever Lectures in Geometry. Princeton 
University Press, Princeton, NJ, 1993.

\bibitem[L]{lazarsfeld} 
R.~Lazarsfeld, {\em{Positivity in algebraic geometry. I. Classical 
setting: line bundles and linear series}}, 
Ergebnisse der Mathematik und ihrer Grenzgebiete. 3. Folge. A 
Series of Modern Surveys in Mathematics [Results in 
Mathematics and Related Areas. 3rd Series. A Series 
of Modern Surveys in Mathematics], \textbf{48}. Springer-Verlag, Berlin, 2004. 

\bibitem[M]{matsuki} 
K.~Matsuki, {\em{Introduction to the Mori program}}, 
Universitext. Springer-Verlag, New York, 2002. 

\bibitem[O1]{oda1}
T.~Oda, {\em{Torus embeddings and applications}}, 
Based on joint work with Katsuya Miyake. Tata 
Institute of Fundamental Research Lectures on 
Mathematics and Physics, \textbf{57}. Tata Institute 
of Fundamental Research, Bombay; Springer-Verlag, 
Berlin-New York, 1978.

\bibitem[O2]{oda2}
T.~Oda, {\em{Convex bodies and algebraic geometry. An introduction to 
the theory of toric varieties}}, 
Translated from the Japanese. Ergebnisse 
der Mathematik und ihrer Grenzgebiete (3) [Results 
in Mathematics and Related Areas (3)], \textbf{15}. Springer-Verlag, 
Berlin, 1988.

\bibitem[Pan]{pang}
T.-S.~Pang, The Harder--Narasimhan filtrations 
and rational contractions, Ph.D. thesis, Universit\"at Freiburg, 2015.

\bibitem[Par]{park}
H.~S.~Park, 
The Chow rings and GKZ-decompositions for $\mathbb Q$-factorial 
toric varieties, 
Tohoku Math. J. (2) \textbf{45} (1993), no. 1, 109--145.

\bibitem[Pay]{payne-m}
S.~Payne, 
Moduli of toric vector bundles, 
Compos. Math. \textbf{144} (2008), no. 5, 1199--1213. 

\bibitem[R]{reid}
M.~Reid, 
Decomposition of toric morphisms, 
{\em{Arithmetic and geometry, Vol. II}}, 395--418, 
Progr. Math., \textbf{36}, Birkh\"auser Boston, Boston, MA, 1983. 

\bibitem[S]{spicer} 
C.~Spicer, Higher-dimensional foliated Mori theory, 
Compos. Math. \textbf{156} (2020), no. 1, 1--38.
 
\bibitem[W]{wang}
W.~Wang, 
Toric foliated minimal model program, 
J. Algebra \textbf{632} (2023), 70--86. 

\end{thebibliography}
\end{document}